\documentclass{gLMA2e}

\usepackage{epstopdf}
\usepackage{subfigure}

\usepackage{graphicx,amsfonts,amsmath,amssymb,amsthm,tabularx} 
\usepackage{bm} 
\usepackage{color}
\usepackage{csquotes}

\theoremstyle{plain}

\newtheorem{theorem}{Theorem}

\newtheorem{lemma}{Lemma}
\newtheorem{corollary}{Corollary}
\newtheorem{conjecture}{Conjecture}
\newtheorem{remark}{Remark}
\theoremstyle{definition} 

\numberwithin{equation}{section}
\numberwithin{figure}{section}
\numberwithin{table}{section}

\begin{document}


\title{On the Approximation of Laplacian Eigenvalues in Graph Disaggregation}

\author{
\name{ Xiaozhe Hu\textsuperscript{a}, John C. Urschel\textsuperscript{b,c,d$\ast$}\thanks{$^\ast$Corresponding author. Email: urschel@mit.edu},
Ludmil T. Zikatanov \textsuperscript{d,e}
}
\affil{ \textsuperscript{a}Department of Mathematics, Tufts University, Medford, MA, USA
\textsuperscript{b}Baltimore Ravens, NFL, Owings Mills, MD, USA;
  \textsuperscript{c}Department of Mathematics, Massachusetts Institute of Technology,
  Cambridge, MA, USA.
  \textsuperscript{d}Department of Mathematics, Penn State University,
  University Park, PA, USA;
  \textsuperscript{e}Institute for Mathematics and 
  Informatics, Bulgarian Academy of Sciences, Sofia, Bulgaria
}  
}

\maketitle

\begin{abstract}
Graph disaggregation is a technique used to address the high cost of computation for power law graphs on parallel processors. The few high-degree vertices are broken into multiple small-degree vertices, in order to allow for more efficient computation in parallel. In particular, we consider computations involving the graph Laplacian, which has significant applications, including diffusion mapping and graph partitioning, among others. We prove results regarding the spectral approximation of the Laplacian of the original graph by the Laplacian of the disaggregated graph. In addition, we construct an alternate disaggregation operator whose eigenvalues interlace those of the original Laplacian. Using this alternate operator, we construct a uniform preconditioner for the original graph Laplacian.
\end{abstract}


\begin{classcode} 05C85; 65F15; 65F08; 68R10 \end{classcode}

\section{Introduction}
A variety of real-world graphs, including web networks
\cite{MR2091634}, social networks \cite{MR2282139}, and bioinformatics
networks \cite{Ji04agraph}, exhibit a degree power law. Namely, the
fraction of nodes of degree $k$, denoted by $P(k)$, follows a power
distribution of the form $P(k) \sim k^{- \gamma}$, where $\gamma$ is
typically in the range $2< \gamma <3$. Networks of this variety are
often referred to as scale-free networks. The pairing of a few
high-degree vertices with many low-degree vertices on large scale-free networks makes 
computations such as Laplacian matrix-vector products and solving linear and eigenvalue equations challenging. The computation of the
minimal nontrivial eigenpair can become prohibitively expensive. This eigenpair has many important applications, such as diffusion mapping and graph partitioning
\cite{CPE:CPE4330060203,Nadler05diffusionmaps,urschel2015cascadic,urschel2014spectral}.

Breaking the few high degree nodes into multiple smaller degree nodes
is a way to address this issue,
especially when large-scale parallel computers are available. This
technique, called graph disaggregation, was introduced by Kuhlemann and Vassilevski \cite{Kuhlemann.V;Vassilevski.P2013a,d2015compatible}. In  
this process, each of the high-degree vertices of the network is
replaced by a graph, such as a cycle or a clique, where each incident
edge of the original node now connects to a node of the cycle or
clique (see Figure~\ref{fig:dis}). 

Independently, Lee, Peng, and Spielman investigated the
concept of graph disaggregation, referred to as vertex splitting, in the setting of combinatorial
spectral sparsifiers ~\cite{lee2015sparsified}. They proved results for graphs
disaggregated from complete graphs and expanders, and used the Schur
complement of the disaggregated Laplacian with respect to the
disaggregated vertices to approximate the original Laplacian.  The
basic motivating assumption in such constructions is that the spectral
structure of the graph Laplacian induced by the disaggregated graph
approximates the spectral structure of the original graph well.  

In \cite{Kuhlemann.V;Vassilevski.P2013a,d2015compatible} Kuhlemann and Vassilevski took a numerical approach. We extend, expand upon, and prove precise and rigorous theoretical results regarding this technique. First, we look at the case of a single disaggregated vertex and establish bounds on the error in spectral approximation with respect to the Laplacians of the original and disaggregated graph, as well as results related to the Cheeger constant. We investigate a conjecture made in \cite{Kuhlemann.V;Vassilevski.P2013a} and give strong theoretical evidence that it does not hold in general. Then, we treat the more general case of disaggregation of multiple vertices and prove analogous results. Finally, we construct an alternative disaggregation operator whose eigenvalues interlace with those of the original graph
Laplacian, and, hence, provide excellent approximation to the
spectrum of the latter. We then use this new disaggregation operator to construct a uniform preconditioner for the graph
Laplacian of the original graph. We prove that the
preconditioned graph Laplacian can be made arbitrarily close to the identity operator 
if we require that the weights of the internal disaggregated edges are
sufficiently large.

 \begin{figure}[]
\begin{center}
\includegraphics[width=0.3\textwidth]{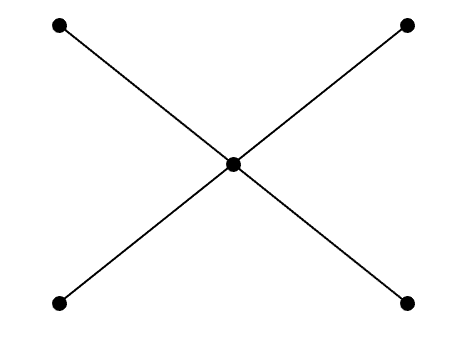}
\includegraphics[width=0.3\textwidth]{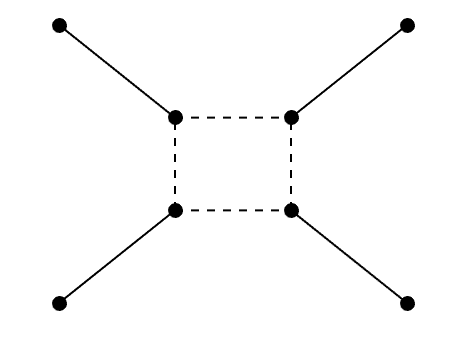}
\includegraphics[width=0.3\textwidth]{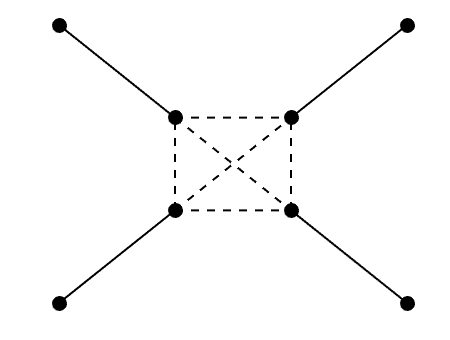}
\caption{Example of disaggregation: original graph (left); disaggregate using cycle (middle); disaggregate using clique (right).}
\label{fig:dis}
\end{center}
\end{figure}

\section{Single Vertex Disaggregation}

Consider a weighted, connected, undirected graph
$\mathsf{G}=(V, E, \omega)$, $| V | = n$. Let $e = (i,j)$ denote an
edge that connects vertices $i$ and $j$, and
$\langle \cdot, \cdot \rangle$ and $\| \cdot \|$ denote the standard
$\ell^2$-inner product and the corresponding induced norm. The
associated weighted graph Laplacian $A \in \mathbb{R}^{n\times n}$ is
given by
\[
\langle A \bm{u} ,\bm{v} \rangle= \sum_{e=(i,j)\in E} \omega_e(u_i-u_j)(v_i-v_j), \quad
\omega_e = (-a_{ij}), 
\]
where we denote the $(i,j)$-{th} element of $A$ by $a_{ij}$.  Without
loss of generality, let us disaggregate the last vertex $\mathsf{v}_n$
of the graph $\mathsf{G}$. Then, the Laplacian can be
written in the following block form
$$
A = 
\begin{pmatrix} 
A_0 & - \bm{a}_{n} \\ 
 - \bm{a}_{n}^T & a_{nn} 
\end{pmatrix},
$$
where $a_{nn}$ is the degree of $\mathsf{v}_n$. Here, we assume that the graph is simply connected and the associated Laplacian $A$ has eigenvalues 
$$
0 = \lambda_1(A) < \lambda_2 (A) \le \cdots \le \lambda_n (A)
$$ 
and corresponding eigenvectors 
$$
\bm{1}_n=\bm{\varphi}^{(1)}(A) , \bm{\varphi}^{(2)}(A), \cdots, \bm{\varphi}^{(n)}(A), \qquad \text{where} \ \bm{1}_n=(\underbrace{1,\cdots,1}_n)^T.
$$
The eigenpair $\left( \lambda_2(A), \bm{\varphi}^{(2)}(A) \right)$ has special significance, and therefore $\lambda_2(A)$ is referred to as the {\it algebraic connectivity}, denoted $a(\mathsf{G})$, and $\bm{\varphi}^{(2)}(A)$ is referred to as the {\it Fiedler vector}.

We can also write a given nontrivial eigenpair $(\lambda(A), \bm{\varphi}(A))$, $\lambda(A) \ne 0$, $\| \bm{\varphi}(A) \| =1$, in block notation, namely
$$ 
\bm{\varphi}(A) = 
\begin{pmatrix} 
\bm{\varphi}_0 \\ 
\varphi_n 
\end{pmatrix}.
$$
We have the relations
 \begin{eqnarray*} 
 \langle \bm{\varphi}_0, \bm{1}_{n_0} \rangle + \varphi_n &=& 0,  \\
  A_0 \bm{\varphi}_0 - \varphi_n \bm{a}_{n} &=& \lambda (A) \bm{\varphi}_0, \\
a_{nn} \varphi_n  - \bm{a}_{n}^T \bm{\varphi}_0 &= &\lambda (A) \varphi_n,
\end{eqnarray*}
where $n_0 = n-1$. Suppose that the vertex $\mathsf{v}_n$ is disaggregated into $d$ vertices, with an unspecified connected structure between the disaggregated elements. We will denote this graph by $\mathsf{G}_D$. This induces a disaggregated graph Laplacian $A_D \in \mathbb{R}^{N \times N}$, $N = n_0 + d$, with eigenvalues $0 = \lambda_1(A_D) < \lambda_2 (A_D) \le \cdots \le \lambda_N (A_D)$ and corresponding eigenvectors $\bm{1}_N = \bm{\varphi}^{(1)}(A_D) , \bm{\varphi}^{(2)}(A_D), \cdots, \bm{\varphi}^{(N)}(A_D)$. We can write $A_D$ in block form
$$ A_D = 
\begin{pmatrix} 
A_0 & - A_{0n} \\ 
-A_{0n}^T & A_n \end{pmatrix}.$$
We have the relations
\begin{align*}
 a_{nn} &= \bm{a}_{n}^T \bm{1}_{n_0}, \\
 A_{0n}^T \bm{1}_{n_0} &=   A_n \bm{1}_d, \\
 \bm{a}_{n} &= A_0 \bm{1}_{n_0} = A_{0n} \bm{1}_d.
\end{align*}
Let us introduce the prolongation operator $P: \mathbb{R}^n \rightarrow \mathbb{R}^N$,
\begin{equation} \label{def:P}
P = \begin{pmatrix}
I_{n_0 \times n_0}  & 0\\
0 & \bm{1}_d
\end{pmatrix}.
\end{equation}

The following result is immediate.

\begin{lemma}\label{immediate}
Let $A$ and $A_D$ be the graph Laplacian of the original graph $\mathsf{G}$ and the disaggregated and simply connected graph $\mathsf{G}_D$, respectively.  If $P$ is defined as \eqref{def:P}, then we have
$$A = P^T A_D P.$$
\end{lemma}

We aim to show that the algebraic connectivity of $A_D$ is bounded away from the algebraic connectivity of the original graph $A$. To do so, suppose we have an eigenpair $(\lambda, \bm{\varphi})$ of the Laplacian of the original graph $\mathsf{G}$. We prolongate the eigenvector $\bm{\varphi}$ to the disaggregated graph $\mathsf{G}_D$ and obtain an approximate eigenvector by the procedure
\begin{equation}\label{def:hat-varphi}
\widetilde{\bm{\varphi}} = P \bm{\varphi} - s \bm{1}_N = 
\begin{pmatrix}
\bm{\varphi}_0 \\
\varphi_n \bm{1}_d
\end{pmatrix}
- s \bm{1}_N, \quad \text{where} \ s = \frac{d-1}{N} \varphi_n.
\end{equation}
This gives $\langle \widetilde{\bm{\varphi}}, \bm{1}_N \rangle = 0$.


We consider $\widetilde{\bm{\varphi}}$ to be an approximation of
$\bm{\varphi}$ on the non-trivial eigenspace of the disaggregated
operator $A_D$. We have the following relation between the eigenvalue
$\lambda$ of $A$ and the Rayleigh quotient of
$\widetilde{\bm{\varphi}}$ with respect to $A_D$.

\begin{lemma} \label{lem:rq-relation}
Let $(\lambda, \bm{\varphi})$, $\|\bm{\varphi}\| = 1$, be an eigenpair of the graph Laplacian $A$ associated with a simply connected graph $\mathsf{G}$, and $\widetilde{\bm{\varphi}}$ be defined by \eqref{def:hat-varphi}. We have
$$\mathrm{RQ}(\widetilde{\bm{\varphi}}):=\frac{ \langle  A_D \widetilde{\bm{\varphi}} , \widetilde{\bm{\varphi}} \rangle}{  \langle\widetilde{\bm{\varphi}} , \widetilde{\bm{\varphi}} \rangle } = \frac{ \lambda } { 1 + \tfrac{(d -1)n}{N} \varphi_n^2 } .$$
\end{lemma}

\begin{proof}
We have
\begin{align*}
 \langle  \widetilde{\bm{\varphi}} , \widetilde{\bm{\varphi}} \rangle &= \langle P \bm{\varphi}  - s \bm{1}_N,  P \bm{\varphi}  - s \bm{1}_N \rangle  = \langle P \bm{\varphi} ,  P \bm{\varphi} \rangle - 2s  \langle P \bm{\varphi} , \bm{1}_N \rangle + s^2 \langle \bm{1}_N, \bm{1}_N  \rangle \\
 & = \langle \bm{\varphi_0}, \bm{\varphi}_0 \rangle + \varphi_n^2 \langle  \bm{1}_d, \bm{1}_d \rangle - 
 2s \left(  \langle \bm{\varphi}_0, \bm{1}_{n_0} \rangle  + \varphi_n \langle \bm{1}_d, \bm{1}_d  \rangle \right) + s^2N   \\
 & =  \langle \bm{\varphi_0}, \bm{\varphi}_0 \rangle + \varphi_n^2 + (d-1) \varphi_n^2  - 2s (d-1) \varphi_n + s^2 N  \\
 & = 1 + \left[ (d-1) - 2\frac{(d-1)^2}{N} + \frac{(d-1)^2}{N} \right] \varphi_n^2  \\
 & = 1 + \frac{(d-1)n}{N} \varphi_n^2 
 \end{align*}
 and
  \begin{eqnarray*}
  \langle {A_D}  \widetilde{\bm{\varphi}} , \widetilde{\bm{\varphi}}\rangle&=&
\langle  A_D (P \bm{\varphi}  - s \bm{1}_N),  P \bm{\varphi}  - s \bm{1}_N \rangle \\
&=& \langle A_D P \bm{\varphi} , P \bm{\varphi}\rangle - 2s\langle A_D  \bm{1}_N  , P\bm{\varphi}\rangle + s^2 \langle A_D \bm{1}_N, \bm{1}_N  \rangle  \\
  & = & \langle P^TA_D P \bm{\varphi} , \bm{\varphi}\rangle  = \langle A \bm{\varphi}, \bm{\varphi} \rangle = \lambda.
 \end{eqnarray*}
This completes the proof.
\end{proof}

The following result quickly follows by applying Lemma \ref{lem:rq-relation} to the Fielder vector.

\begin{theorem} \label{thm:algebraic-connectivity}
Let $\bm{\varphi} = (\bm{\varphi}_0, \varphi_n)^T$, $\| \bm{\varphi} \| = 1$, be the Fiedler vector of the graph Laplacian $A$ associated with a simply connected graph $\mathsf{G}$. Let $A_D$ be the graph Laplacian corresponding to the disaggregated and simply connected graph $\mathsf{G}_D$ resulting from disaggregating one vertex into $d>1$ vertices. We have
$$\frac{a(\mathsf{G})}{a(\mathsf{G}_D)} \ge  1 + \frac{(d -1)n}{N} \varphi_n^2. $$
\end{theorem}

\begin{proof}
Noting that $\widetilde{\bm{\varphi}}$ is orthogonal to $\bm{1}_N$, we have
\begin{align*}
a(\mathsf{G}_D) & \leq \frac{ \langle  A_D \widetilde{\bm{\varphi}} , \widetilde{\bm{\varphi}} \rangle}{  \langle\widetilde{\bm{\varphi}} , \widetilde{\bm{\varphi}} \rangle }  =  \frac{ \lambda } { 1 + \tfrac{(d -1)n}{N} \varphi_n^2 }  =  \frac{ a(\mathsf{G}) } { 1 + \tfrac{(d -1)n}{N} \varphi_n^2 },
\end{align*}
which completes the proof.
\end{proof}

If the characteristic value of the disaggregated vertex is non-zero,
then the algebraic connectivity of the disaggregated graph stays
bounded away from that of the original graph, independent of the
structure of $A_n$. Therefore, as the weight on the internal edges
approaches infinity, the approximation stays bounded away. 

In \cite{Kuhlemann.V;Vassilevski.P2013a}, the authors made the following conjecture.

\begin{conjecture} \label{wrongconject}
Under certain conditions the Laplacian eigenvalues of the graph
  Laplacian of the disaggregated graph approximate the eigenvalues of
  the graph Laplacian of the original graph, provided that the weight
  on the internal edges of the disaggregation is chosen to be large
  enough.
\end{conjecture}

Theorem \ref{thm:algebraic-connectivity} directly implies that Conjecture \ref{wrongconject} is false when the characteristic value of the disaggregated vertex is non-zero, which, for a random power law graph, occurs with probability one.

We also have the following result, providing an estimate of how close the 
approximation $\widetilde{\bm{\varphi}}$ is 
to the invariant subspace with respect to $A_D$.
\begin{lemma}
Let $A$ and  $A_D$ be the original graph and disaggregated graph, $(\lambda, \bm{\varphi})$ be an eigenpair of the graph Laplacian $A$ associated with a simply connected graph $\mathsf{G}$, and $\widetilde{\bm{\varphi}}$ be defined by \eqref{def:hat-varphi}. We have
$$
\| A_D \widetilde{\bm{\varphi}} -  \mathrm{RQ}(\widetilde{\bm{\varphi}} ) \widetilde{\bm{\varphi}} \| \le  \left( \|A_{0n}^T (\bm{1}_{n_0} - \bm{\varphi}_0 / \varphi_n ) \| + \frac{\sqrt{dn (d+n)}}{N} \lambda + \frac{d n}{N} \lambda | \varphi_n | \right) | \varphi_n |.$$
\end{lemma}

\begin{proof}
We recall that 
$$ \| \widetilde{\bm{\varphi}} \| = \left( 1 + \frac{(d -1)n}{N} \varphi_n^2 \right)^{1/2}$$
and
$$ \lvert \lambda - \mathrm{RQ}(\widetilde{\bm{\varphi}}) \rvert  = \frac{ \frac{(d -1)n}{N} \varphi_n^2  }{ 1 + \frac{(d -1)n}{N} \varphi_n^2} \lambda.$$
We also have
\begin{align*}
A_D \widetilde{\bm{\varphi}}&=  A_D \left(P \bm{\varphi} - s\bm{1}_N \right) = A_D P \bm{\varphi} =  \begin{pmatrix} A_0 \bm{\varphi}_0 - \varphi_n A_{0n} \bm{1}_d  \\ \varphi_n A_n \bm{1}_d - A_{0n}^T \bm{\varphi}_0  \end{pmatrix} =  \begin{pmatrix} A_0 \bm{\varphi}_0 - \varphi_n \bm{a}_n \\ A_{0n}^T ( \varphi_n \bm{1}_{n_0} - \bm{\varphi}_0 )  \end{pmatrix} \\
&= \lambda P \bm{\varphi}  + \begin{pmatrix} \bm{0} \\ A_{0n}^T ( \varphi_n \bm{1}_{n_0} - \bm{\varphi}_0 ) - \lambda  \varphi_n \bm{1}_d \end{pmatrix} = \lambda \widetilde{\bm{\varphi}} + s \lambda \bm{1}_N + \begin{pmatrix} \bm{0} \\ A_{0n}^T ( \varphi_n \bm{1}_{n_0} - \bm{\varphi}_0 ) - \lambda  \varphi_n \bm{1}_d \end{pmatrix}\\
& = \lambda \widetilde{\bm{\varphi}} + \frac{\lambda \varphi_n}{N} \left[  (d-1) 
\begin{pmatrix}
\bm{1}_0 \\ \bm{0}
\end{pmatrix}
- 
n
 \begin{pmatrix}
\bm{0} \\ \bm{1}_d
\end{pmatrix}
   \right]
   +  \begin{pmatrix} \bm{0} \\ A_{0n}^T ( \varphi_n \bm{1}_{n_0} - \bm{\varphi}_0 ) \end{pmatrix},
\end{align*}
giving
\begin{align*}
\| A_D \widetilde{\bm{\varphi}} -  \mathrm{RQ}(\widetilde{\bm{\varphi}} ) \widetilde{\bm{\varphi}} \|  &\le \| A_D \widetilde{\bm{\varphi}} - \lambda \widetilde{\bm{\varphi}} \| +  \vert \lambda -   \mathrm{RQ}(\widetilde{\bm{\varphi}})  \rvert \| \widetilde{\bm{\varphi}} \|   \\
&\le  \|  A_{0n}^T (\varphi_n \bm{1}_{n_0} - \bm{\varphi}_0)  \| +\frac{ \sqrt{(d-1)^2 n_0 + d n^2}}{N}  | \varphi_n | \lambda + \frac{ \frac{(d -1)n}{N}  }{ \sqrt{1 + \frac{(d-1)n}{N} \varphi_n^2}} \varphi^2_n \lambda \\
&\le \left( \|A_{0n}^T (\bm{1}_{n_0} - \bm{\varphi}_0 / \varphi_n ) \| + \frac{\sqrt{dn(d+n)}}{N} \lambda + \frac{d n}{N} \lambda | \varphi_n | \right) | \varphi_n |.
\end{align*}
\end{proof}

In many applications, we are only concerned with minimal Laplacian eigenpairs. For minimal eigenvalues of scale-free graphs, we have $\lambda = O(1)$ and $\varphi_n = O (N^{-1/2})$. In this way, often the largest source of error comes from the term $\|A_{0n}^T (\varphi_n \bm{1}_{n_0} - \bm{\varphi}_0)\| $. Heuristically, the error of this term is typically best controlled when $d$ is relatively small and each new disaggregate is connected to roughly the same number of exterior vertices.

Next, we consider the eigenvalues of the normalized Laplacian. This gives us insight into how the Cheeger constant changes after disaggregating a vertex.  Again, suppose we have an eigenpair $(\nu, \bm{\phi})$ of the normalized graph Laplacian $D^{-1}A$, where $D$ is the degree matrix of $\mathsf{G}$, namely $D = \text{diag}(a_{11}, a_{22}, \cdots, a_{nn})$.
Again, we prolongate the eigenvector $\bm{\phi}$ and obtain an approximate eigenvector of the disaggregated normalized graph Laplacian $D_D^{-1}A_D$, where $D_D$ is the degree matrix of $\mathsf{G}_D$, in a similar fashion
\begin{equation}\label{def:phi}
\widetilde{\bm{\phi}} = P \bm{\phi} - s \bm{1}_N,  \quad \text{where} \ s = \frac{\langle P\bm{\phi}, \bm{1}_N \rangle_{D_D}}{\langle \bm{1}_N, \bm{1}_N \rangle_{D_D}}.
\end{equation}
We may write $D_D$ in the following way
\begin{align*}
D_D &= \text{diag}(a^D_{1,1}, \cdots, a^D_{n_0,n_0}, a^D_{n,n}, a^D_{n+1,n+1}, \cdots, a^D_{N,N})  \\
& = \text{diag}(a^D_{1,1}, \cdots, a^D_{n_0,n_0}, \omega_{n}, \omega_{n+1}, \cdots, \omega_N) + \text{diag}(0, \cdots, 0, d^{ex}_n, d^{ex}_{n+1}, \cdots, d^{ex}_{N}) \\
& =: D_D^1 + D_D^{ex},
\end{align*}
where $\omega_n, \omega_{n+1}, \cdots \omega_N$ are the weights of the edges incident with vertex $\mathsf{v}_n$ on the original graph (note $\sum_{i=n}^N \omega_i = a_{n,n}$) and $d^{ex}_i = a^D_{i,i} - \omega_i$, $i = n, n+1, \cdots, N$. We may also rewrite the shift $s$ as
$$s = \frac{\langle P\bm{\phi}, \bm{1}_N \rangle_{D_D}}{\langle \bm{1}_N, \bm{1}_N \rangle_{D_D}} = \frac{\sum_{i = n}^N  d^{ex}_i  \phi_n}{\sum_{i=1}^N a^D_{i,i}}.$$

Let $\omega_{\text{total}}(\mathsf{H})$ denote the total weights of a graph $\mathsf{H}$, and let $\mathsf{G}_a$ be the disaggregated local subgraph.  Similarly,  we consider $\widetilde{\bm{\phi}}$ as an approximation of the eigenvectors of the disaggregated normalized graph Laplacian. We have the following lemma.  
\begin{theorem} \label{thm:eigenvalue_normalized}
Let $(\nu, \bm{\phi})$ be an eigenpair of the normalized graph Laplacian associated with a simply connected graph $\mathsf{G}$ and $\widetilde{\bm{\phi}}$ be defined by \eqref{def:phi}. We have
\begin{equation}\label{eqn:RQ_normalized}
\frac{\langle A_D \widetilde{\bm{\phi}}, \widetilde{\bm{\phi}} \rangle}{\langle D_D \widetilde{\bm{\phi}}, \widetilde{\bm{\phi}}  \rangle } = \frac{\nu}{1 + \frac{2 \omega_{\text{total}}(\mathsf{G}) \, \omega_{\text{total}}(\mathsf{G}_a)}{\omega_{\text{total}}(\mathsf{G}_D)} \phi_n^2 },
\end{equation}
and
\begin{equation}\label{ine:nu_2}
\nu^D_2  = \alpha \nu_2, \quad \alpha = \left(    1 + \frac{2 \omega_{\text{total}}(\mathsf{G}) \, \omega_{\text{total}}(\mathsf{G}_a)}{\omega_{\text{total}}(\mathsf{G}_D)} \phi_n^2  \right)^{-1}  \leq 1.
\end{equation}
\end{theorem}
\begin{proof}
We have
\begin{align*}
\langle A_D \widetilde{\bm{\phi}}, \widetilde{\bm{\phi}} \rangle = \langle  A_D(P \bm{\phi} - s \bm{1}_N ), P \bm{\phi} - s \bm{1}_N  \rangle = \langle  A_D P \bm{\phi}, \bm{\phi} \rangle = \langle A \bm{\phi}, \bm{\phi}  \rangle = \nu
\end{align*}
and
\begin{align*}
\langle D_D \bm{\phi}, \bm{\phi}  \rangle & = \langle D_D P \bm{\phi}, P\bm{\phi} \rangle - s \langle  D_D \bm{1}_N, P \bm{\phi} \rangle \\
& = \langle D_D^1 P \bm{\phi}, P\bm{\phi} \rangle + \langle D_D^{ex} P \bm{\phi}, P\bm{\phi} \rangle - s \langle  D_D \bm{1}_N, P \bm{\phi} \rangle \\
& = \langle D \bm{\phi}, \bm{\phi} \rangle + \sum_{i=n}^{N} d_i^{ex} \phi_n^2 - \frac{\left(  \sum_{i=n}^N d_i^{ex} \phi_n \right)^2}{\sum_i^{N} a^D_{i,i}} \\
& = 1 + \frac{ \left( \sum_i^N a^D_{i,i} - \sum_{i=n}^N d_i^{ex}  \right)\sum_{i=n}^N d_i^{ex}  }{\sum_i^N a^D_{i,i}} \phi_n^2 \\
& = 1 + \frac{ \left( \sum_{i=1}^n a_{i,i} \right)  \left( \sum_{i=n}^N d_i^{ex} \right) }{\sum_{i=1}^N a^D_{i,i}} \phi_n^2.
\end{align*}
Noting that $\sum_{i=1}^N a_{i,i} = 2 \omega_{\text{total}}(\sf{G})$, $\sum_{i=1}^N a^D_{i,i} = 2 \omega_{\text{total}}(\mathsf{G}_D)$, and $\sum_{i=n}^N d_i^{ex} = 2\omega_{\text{total}}(\mathsf{G}_{a})$, we obtain \eqref{eqn:RQ_normalized}.  Moreover, \eqref{ine:nu_2} follows directly from \eqref{eqn:RQ_normalized}.
\end{proof}

The Cheeger constant of a weighted graph is defined as follows \cite{Friedland.S;Nabben.R2002a,Chung.F1997a}
\begin{equation*} 
h(\mathsf{G}) = \min_{\emptyset \neq U \subset V} \frac{|  E(U, \bar{U}) |}{\min( \text{vol}(U), \text{vol}(\bar{U}) )}, 
\end{equation*}
where
\begin{equation*}
\bar{U} = V \backslash U, \quad \text{vol}(U) = \sum_{i \in U} \delta_i, \quad \delta_i = \sum_{e=(i,j) \in E} \omega_e, \quad | E(U, \bar{U}) | = \sum_{\substack{e=(i,j)\in E \\ i \in U, j \in \bar{U}}} \omega_e.
\end{equation*}
Due to the Cheeger inequality \cite{Friedland.S;Nabben.R2002a,Chung.F1997a}, the Cheeger constant $h(\mathsf{G})$ and $\nu_2$ are related as follows
\begin{equation}\label{ine:cheeger}
1 - \sqrt{1 - h(\mathsf{G})^2} \leq \nu_2 \leq 2 h(\mathsf{G}).
\end{equation}

\begin{theorem}\label{thm:Cheeger constant}
For the Cheeger constant of the original graph $\mathsf{G}$ and the disaggregated and simply connected graph $\mathsf{G}_D$, we have
\begin{equation*}
h(\mathsf{G}_D) \leq \sqrt{1 - (1 - 2  \alpha h(\mathsf{G}))^2 },
\end{equation*}
where $\alpha$ is defined by \eqref{ine:nu_2}. If $h(\mathsf{G}) \geq \frac{4 \alpha}{ 4 \alpha^2 + 1}$, then $h(\mathsf{G}_D) \leq h(\mathsf{G})$.
\end{theorem}

\begin{proof}
Based on \eqref{ine:cheeger} and \eqref{ine:nu_2}, we have 
\begin{align*}
h(\mathsf{G}_D) \leq \sqrt{1 - (1 - \nu^D_2)^2} = \sqrt{1 - (1- \alpha \nu_2)^2} \leq \sqrt{1 - (1 - 2 \alpha h(\mathsf{G}))^2 }.
\end{align*}
Basic algebra shows that $h(\mathsf{G}_D) \leq h(\mathsf{G})$ if $h(\mathsf{G}) \geq \frac{4 \alpha}{ 4 \alpha^2 + 1}$.
\end{proof}

\section{Graph Disaggregation}
We now move on to the more general case of multiple disaggregated vertices. Without loss of generality, for a graph Laplacian $A \in \mathbb{R}^{n \times n}$, suppose we are disaggregating the first $m$ vertices. This gives us the disaggregated Laplacian $A_D \in \mathbb{R}^{N\times N}$. Here, $N$ is
the number of vertices in the disaggregated graph, given by
\[
N = n - m + \sum_{k=1}^m d_k = n-m+n_d, \quad 
n_d=\sum_{k=1}^m d_k.
\]
Note that we have $m$ groups of vertices associated with the
disaggregation, which we can also number consecutively
\begin{equation}\label{numb}
\{1,\ldots,N\} = \{\underbrace{1,\ldots, d_1}_{d_1},\ldots
\underbrace{d_1+1,\ldots, d_1+d_2}_{d_2},\ldots, n_d+1,\ldots,N\}.
\end{equation}

Similar to the case of a single disaggregated vertex, we can establish a relationship between $A$ and $A_D$ through a prolongation matrix $P: \mathbb{R}^n \rightarrow  \mathbb{R}^N$, given by
\begin{equation}
P = \begin{pmatrix} \label{def:Prolongation}
P_m & 0\\
0 & I_{n_0 \times n_0}
\end{pmatrix}, \quad\mbox{where}\quad n_0=(n-m).
\end{equation}
Here, $P_m\in \mathbb{R}^{n_d \times m}$, and
\[
P_m=\begin{pmatrix}
\bm{1}_{d_1} & 0 & \ldots & 0 \\
0 & \bm{1}_{d_2} & \ldots & 0 \\
\vdots & \vdots & \vdots & \vdots\\
0 & 0 & \ldots & \bm{1}_{d_m}
\end{pmatrix}.
\]
Note that $P_m^T P_m = \operatorname{diag}(d_1,\ldots d_m)$. We have the following lemma, which can be easily verified by simple algebraic calculation.

\begin{lemma} 
Let $A$ and $A_D$ be the graph Laplacian of the original graph $\mathsf{G}$ and the disaggregated and simply connected graph $\mathsf{G}_D$, respectively. If $P$ is defined as \eqref{def:Prolongation}, then we have
\begin{equation} \label{eqn:A-A_D}
A = P^T A_D P.
\end{equation}
\end{lemma}

If we look at the disaggregated graph Laplacian $A_D$ directly, we can obtain a similar bound on the algebraic connectivity as shown in Theorem \ref{thm:algebraic-connectivity}.  Let $(\lambda, \bm{\varphi})$ be an eigenpair of $A$. We can define an approximated eigenvector of $A_D$ by prolongating $\bm{\varphi}$ as follows
\begin{equation} \label{def:hat-varphi-multi}
\widetilde{\bm{\varphi}} = P \bm{\varphi} - s \bm{1}_N, \quad \text{where} \ s = \frac{1}{N}\sum_{i=1}^m (d_i - 1) \varphi_i.
\end{equation}
It is easy to check that $\langle \widetilde{\bm{\varphi}}, \bm{1}_N \rangle = 0$.  Now we have the following lemma about the Rayleigh quotient of $\widetilde{\bm{\varphi}}$ with respect to $A_D$.

\begin{lemma} \label{lem:rq-relation-multi}
Let $(\lambda, \bm{\varphi})$ be an eigenpair of the graph Laplacian $A$ associated with a simply connected graph $\mathsf{G}$ and $\widetilde{\bm{\varphi}}$ be defined by \eqref{def:hat-varphi-multi}. We have
$$\mathrm{RQ}(\widetilde{\bm{\varphi}}):=\frac{ \langle  A_D \widetilde{\bm{\varphi}} , \widetilde{\bm{\varphi}} \rangle}{  \langle\widetilde{\bm{\varphi}} , \widetilde{\bm{\varphi}} \rangle } \leq \frac{ \lambda } { 1 + \frac{1}{N}\sum_{i=1}^m (d_i - 1) (n + n_d - md_i)   \varphi_i^2 }. 
$$
Moreover, if for $d_i^{\max} = \max_i d_i$, we have $n + n_d - m d_i^{\max} > 0$, then
$
RQ(\widetilde{\bm{\varphi}}) <  \lambda.
$
\end{lemma}

\begin{proof}
We note that
\begin{align*}
  \langle {A_D}  \widetilde{\bm{\varphi}} , \widetilde{\bm{\varphi}}\rangle&=\langle  A_D (P \bm{\varphi}  - s \bm{1}_N),  P \bm{\varphi}  - s \bm{1}_N \rangle = \langle A_D P \bm{\varphi} , P \bm{\varphi}\rangle - 2s\langle A_D  \bm{1}_N  , P\bm{\varphi}\rangle + s^2 \langle A_D \bm{1}_N, \bm{1}_N  \rangle  \\
  & = \langle P^TA_D P \bm{\varphi} , \bm{\varphi}\rangle  = \langle A \bm{\varphi}, \bm{\varphi} \rangle = \lambda.
 \end{align*}
 Denoting $\bm{\varphi}$ by $\bm{\varphi} = (\varphi_1, \varphi_2, \cdots, \varphi_m, \bm{\varphi}_0)^T$, we have
 \begin{align*}
 \langle  \widetilde{\bm{\varphi}} , \widetilde{\bm{\varphi}} \rangle &= \langle P \bm{\varphi}  - s \bm{1}_N,  P \bm{\varphi}  - s \bm{1}_N \rangle  = \langle P \bm{\varphi} ,  P \bm{\varphi} \rangle - 2s  \langle P \bm{\varphi} , \bm{1}_N \rangle + s^2 \langle \bm{1}_N, \bm{1}_N  \rangle \\
 & = \langle \bm{\varphi_0}, \bm{\varphi}_0 \rangle + \sum_{i=1}^m \varphi_i^2 \langle  \bm{1}_{d_i}, \bm{1}_{d_i} \rangle - 
 2s \left(  \langle \bm{\varphi}_0, \bm{1}_{n_0} \rangle  + \sum_{i=1}^m \varphi_i \langle \bm{1}_{d_i}, \bm{1}_{d_i}  \rangle \right) + s^2N   \\
 & =  \langle \bm{\varphi_0}, \bm{\varphi}_0 \rangle + \sum_{i=1}^m  \varphi_i^2 + \sum_{i=1}^m(d_i-1) \varphi_i^2  - 2s \sum_{i=1}^m (d_i-1) \varphi_i + s^2 N  \\
 & = 1 + \sum_{i=1}^m(d_i-1) \varphi_i^2 - \frac{1}{N} \left(  \sum_{i=1}^m (d_i -1) \varphi_i  \right)^2 \\
 & \geq 1+  \sum_{i=1}^m(d_i-1) \varphi_i^2 - \frac{m}{N} \sum_{i=1}^m (d_i -1)^2 \varphi_i^2 \\
 & = 1 + \sum_{i=1}^m (d_i - 1) \frac{N - m (d_i - 1)}{N} \varphi_i^2 \\
 & = 1 + \frac{1}{N}\sum_{i=1}^m (d_i - 1) (n + n_d - md_i) \varphi_i^2.
 \end{align*}
This completes the proof. 
\end{proof}

From the Rayleigh quotient and applying the above lemma to the Fielder vector, we have the following theorem concerning the algebraic connectivity. 

\begin{theorem} \label{thm:algebraic-connectivity-multi}
Let $\bm{\varphi}$ be the Fiedler vector of the graph Laplacian $A$ associated with a simply connected graph $\mathsf{G}$ and $A_D$ be the graph Laplacian corresponding to the disaggregated and simply connected graph $\mathsf{G}_D$. Suppose we have disaggregated $m$ vertices and each of those vertices are disaggregated into $d_i>1$ vertices, $i=1,2,\cdots,m$. We have
$$\frac{a(\mathsf{G})}{a(\mathsf{G}_D)} \ge   1 + \frac{1}{N}\sum_{i=1}^m (d_i - 1) (n + n_d - md_i)   \varphi_i^2 .$$
\end{theorem}

\begin{proof}
The proof is similar to the proof of Theorem \ref{thm:algebraic-connectivity} and uses Lemma \ref{lem:rq-relation-multi}.
\end{proof}

It is possible to perform a more careful estimate, using the fact that disaggregating $m$ vertices at once is equivalent to disaggregating $m$ vertices one by one.  Denote $n_0 = n$ and $n_i = n_{i-1} - 1 + d_i$, $i = 1,2,\cdots, m$.  Note that $n_i = n - i + \sum_{k=1}^i d_k$, $i = 1,2,\cdots, m$ and $N = n_m$.  Recursively applying Theorem \ref{thm:algebraic-connectivity}, we have the following result.

\begin{theorem} \label{thm:algebraic-connectivity-multi-1}
Let $\bm{\varphi}$ be the Fiedler vector of the graph Laplacian $A$ associated with a simply connected graph $\mathsf{G}$ and $A_D$ be the graph Laplacian corresponding to the disaggregated and simply connected graph $\mathsf{G}_D$, respectively. Suppose we disaggregated $m$ vertices and each of those are disaggregated into $d_i>1$ vertices, $i=1,2,\cdots,m$. We have
$$\frac{a(\mathsf{G})}{a(\mathsf{G}_D)} \ge   \prod \left( 1 +  \frac{ (d_i - 1) n_{i-1} }{n_i} \varphi_i^2 \right). $$
\end{theorem}
\begin{proof}
Let $\mathsf{G}_D^i$ be the resulting graph after disaggregating the $i$-th vertex in the graph $\mathsf{G}_D^{i-1}$ and note that $\mathsf{G}_D = \mathsf{G}_D^m$. We have
\begin{equation*}
\frac{a(\mathsf{G})}{a(\mathsf{G}_D)} = \frac{\mathsf{a(G)}}{a(\mathsf{G}_D^1)} \times\frac{a(\mathsf{G}_D^1)}{a(\mathsf{G}_D^2)}  \times \cdots \times \frac{a(\mathsf{G}_D^{m-1})}{a(\mathsf{G}_D^m)}.
\end{equation*}
The result follows immediately from applying Theorem \ref{thm:algebraic-connectivity} on each pair of graphs $\mathsf{G}_D^{i}$ and $\mathsf{G}_D^{i+1}$.  
\end{proof}

\begin{remark}
A direct consequence of Theorem \ref{thm:algebraic-connectivity-multi-1} is $a(\mathsf{G}_D) \leq a(\mathsf{G})$.
\end{remark}

Similarly, we also have the following result concerning the minimal eigenvalue of the normalized graph Laplacian after disaggregating several vertices.  Denote $\mathsf{G}_D^0 = \mathsf{G}$, and denote the graph after disaggregating vertex $i$ by $\mathsf{G}_D^i$.  Note that $\mathsf{G}_D^m = \mathsf{G}_D$.  The local subgraph corresponding to disaggregating vertex $i$ is denoted by $\mathsf{G}_a^i$.

\begin{theorem}
Let $\nu_2$ be the second smallest eigenvalue of the normalized graph Laplacian associated with a simply connected graph $\mathsf{G}$ and $\nu_2^D$ be the second smallest eigenvalue of the normalized graph Laplacian corresponding to the disaggregated and simply connected graph $\mathsf{G}_D$. Suppose we disaggregated $m$ vertices and each of them are disaggregated into $d_i>1$ vertices, $i=1,2,\cdots,m$. We have
\begin{equation*}
\frac{\nu_2}{\nu_2^D} \geq \prod_{i=1}^m \left(  1 +  \frac{2 \omega_{\text{total}}(\mathsf{G}_D^{i-1})  \,\omega_{\text{total}}(\mathsf{G}_a^i)}{ \omega_{\text{total}}( \mathsf{G}_D^i)  }\phi_i^2    \right).
\end{equation*}
Consequently, we have $\nu_2^D =\alpha \nu_2$, where 
\begin{equation} \label{def:alpha}
\alpha := \left[   \prod_{i=1}^m \left(  1 +  \frac{2 \omega_{\text{total}}(\mathsf{G}_D^{i-1}) \omega_{\text{total}}(\mathsf{G}_a^i)}{ \omega_{\text{total}}( \mathsf{G}_D^i)  }\phi_i^2    \right) \right]^{-1} \leq 1.
\end{equation}
\end{theorem}

\begin{proof}
The result follows by applying Theorem \ref{thm:eigenvalue_normalized} recursively.
\end{proof}

Based on the estimates on the eigenvalues of normalized graph Laplacian, we can estimate the Cheeger constants as follows.

\begin{theorem}\label{thm:Cheeger constant-1}
For the Cheeger constant of the original graph $\mathsf{G}$ and the disaggregated and simply connected graph $\mathsf{G}_D$, we have
\begin{equation*}
h(\mathsf{G}_D) \leq \sqrt{1 - (1 - 2  \alpha h(\mathsf{G}))^2 },
\end{equation*}
where $\alpha$ is defined by \eqref{def:alpha}. If $h(\mathsf{G}) \geq \frac{4 \alpha}{ 4 \alpha^2 + 1}$, then $h(\mathsf{G}_D) \leq h(\mathsf{G})$.
\end{theorem}

\begin{proof}
The proof is the same as the proof of Theorem \ref{thm:Cheeger constant}.
\end{proof}

\section{Preconditioning Using Disaggregated Graph}

We aim to show eigenvalue interlacing
between $A$ and a new operator, which is obtained by scaling $A_D$
appropriately.

We can rescale $P$ by introducing $\widetilde{P}= D_s P$, where 
\begin{equation}\label{def:D_s}
D_s = \operatorname{diag}(d^{-1/2}_1I_{d_1 \times d_1},\ldots , d^{-1/2}_m I_{d_m \times d_m}, I_{n0 \times n0})
\end{equation}
is a diagonal scaling matrix, giving us
$\widetilde{P}^T\widetilde{P}=I$.  Based on the scaled prolongation,
we are able to show the eigenvalues of diagonal scaled matrix
\begin{equation} \label{def:tildeA_D}
\widetilde{A}_D:= D_s^{-1} A_D D_s^{-1}
\end{equation}
interlaces with $A$.  First, let us recall the interlacing theorem.
 
\begin{theorem}[Interlacing Theorem \cite{Courant.R;Hilbert.D1924a},
  Vol. 1, Chap. I] \label{thm:interlacing} Let
  $S \in \mathbb{R}^{N \times n}$ be such that
  $S^T S = I_{n \times n}$, $n < N$ and let
  $B \in \mathbb{R}^{N \times N}$ be symmetric, with eigenvalues
  $\lambda_1 \le \lambda_2 \le ... \le \lambda_n $. Define
  $A = S^T B S$ and let $A$ have eigenvalues
  $\mu_1\le \mu_2 \le ... \le \mu_m$. Then
  $ \lambda_i \le \mu_i \le \lambda_{n-m+i}.$
\end{theorem}

From here, we have the following.
\begin{theorem}
  Let $A$ have eigenvalues
  $\lambda_1(A) \le \lambda_2(A) \le ... \le \lambda_n(A)$ and
  $\widetilde{A}_D=D_s^{-1}A_DD_s^{-1}$ have eigenvalues
  $\lambda_1(\widetilde{A}_D) \le \lambda_2(\widetilde{A}_D) \le
  ... \le \lambda_N(\widetilde{A}_D)$. Then
$$ \lambda_i(\widetilde{A}_D) \le \lambda_i(A) \le \lambda_{N-n+i}(\widetilde{A}_D). $$
\end{theorem}

\begin{proof}
From the above Lemma, we have
\[
A = P^T A_D P = 
P^T D_s D_s^{-1} A_D D_s^{-1} D_s P = 
\widetilde{P}^TD_s^{-1} A_D D_s^{-1} \widetilde{P} = \widetilde{P}^T \widetilde{A}_D \widetilde{P}.
\]
As $\widetilde{P}^T\widetilde{P}=I_{n \times n}$, by the Interlacing
Theorem \ref{thm:interlacing}, the eigenvalues of $A$ and
$\widetilde{A}_D$ interlace.
\end{proof}

We now discuss how to use the disaggregated graph $\mathsf{G}_D$ to
solve the graph Laplacian on the original graph $\mathsf{G}$.  Here,
we will use $\widetilde{A}_D$ as the auxiliary problem and design a
preconditioner based on the Fictitious Space Lemma
\cite{Nepomnyaschikh.S1992} and auxiliary space framework
\cite{Xu.J1996}.  Because $A$ and $\widetilde{A}_D$ are both symmetric positive
semi-definite, we first state the refined version of the Fictitious
Space Lemma proposed in
\cite{Dios.B;Brezzi.F;Marini.L;Xu.J;Zikatanov.L2014a}.

\begin{theorem}[Theorem 6.3 and 6.4 in \cite{Dios.B;Brezzi.F;Marini.L;Xu.J;Zikatanov.L2014a}]
Let $\widetilde{V}$ and $V$ be two Hilbert spaces and $\Pi: \widetilde{V} \mapsto V$ be a surjective map.  Suppose that $\widetilde{\mathcal{A}}: \widetilde{V} \mapsto \widetilde{V}'$ and $\mathcal{A}: V \mapsto V' $ are symmetric semi-definite operators. Moreover, suppose
\begin{align}
 & \Pi (N( \widetilde{\mathcal{A}})) = N(\mathcal{A}), \label{eqn:null-space} \\
 &\| \Pi \, \widetilde{v} \|_A \leq c_1 \| \widetilde{v} \|_{\widetilde{A}}, \quad \forall \ \widetilde{v} \in \widetilde{V}, \label{ine:upper} \\
& \text{for any} \ v\in V \ \text{there exists} \ \widetilde{v} \in \widetilde{V} \ \text{such that} \ \Pi \, \widetilde{v} = v \ \text{and} \  \| \widetilde{v} \|_{\widetilde{A}} \leq c_0 \| v \|_A, \label{ine:lower}
\end{align}
then for any symmetric positive definite operator $\widetilde{\mathcal{B}}: \widetilde{V}' \mapsto \widetilde{V}$, we have that for $\mathcal{B} = \Pi \, \widetilde{\mathcal{B}} \, \Pi^T$,
\begin{equation*}
\kappa(\mathcal{BA}) \leq \left( \frac{c_1}{c_0} \right)^2 \kappa(\widetilde{\mathcal{B}} \widetilde{\mathcal{A}}).
\end{equation*}
\end{theorem}

Applying the above theory to our disaggregation framework, we take $\mathcal{A} = A$, $\widetilde{\mathcal{A}} = \widetilde{A}_D$, and $\Pi = \widetilde{P}^T$.  Noting that the null space of $\widetilde{A}_D$ is spanned by $D_s \bm{1}_N$, we have
$$
\widetilde{P}^T D_s \bm{1}_N = P^T D_s^2 \bm{1}_N = \bm{1}_n 
$$
which verifies \eqref{eqn:null-space}.  Naturally, $\widetilde{P}^T$ is surjective. Using a preconditioner $\widetilde{B}_D$ of $\widetilde{A}_D$, we can define a preconditioner 
\begin{equation*}
B = \widetilde{P}^T \widetilde{B}_D \widetilde{P}
\end{equation*}
for $A$. We give the following results concerning the quality of the preconditioner $B$.

\begin{corollary}\label{coro:prec-B}
Let $A$ be the graph Laplacian corresponding to the graph $\mathsf{G}$ and $A_D$ be the graph Laplacian corresponding to the disaggregated and simply connected graph $\mathsf{G}_D$. Let $D_s$ be defined by \eqref{def:D_s} and $\widetilde{A}_D$ be defined by \eqref{def:tildeA_D}.  If
\begin{equation} \label{ine:upper-c1}
\| \widetilde{P}^T \widetilde{\bm{v}} \|_A \leq c_1 \| \widetilde{\bm{v}} \|_{\widetilde{A}_D}, \quad \forall \ \widetilde{v} \in \widetilde{V}
\end{equation}
and for any $\bm{v} \in V$,  there exist a $\widetilde{\bm{v}} \in \widetilde{V}$ such that $\widetilde{P}^T \widetilde{\bm{v}} = \bm{v}$ and
\begin{equation} \label{ine:lower-c0}
\| \widetilde{\bm{v}} \|_{\widetilde{A}_D} \leq c_0 \| \bm{v} \|_A.
\end{equation}
Then for the preconditioner $B = \widetilde{P}^T \widetilde{B}_D \widetilde{P}$, we have
\begin{equation*}
\kappa(BA) \leq \left( \frac{c_1}{c_0} \right)^2 \kappa(\widetilde{B}_D \widetilde{A}_D).
\end{equation*}
\end{corollary}

We need to verify that conditions \eqref{ine:upper-c1} and
\eqref{ine:lower-c0} hold for $\widetilde{P} = D_s P$.  For condition
\eqref{ine:lower-c0}, we choose
$\widetilde{\bm{v}} = \widetilde{P}^T \bm{v}$ for any $\bm{v} \in V$,
giving
$\widetilde{P}^T \widetilde{\bm{v}} = \widetilde{P}^T \widetilde{P}
\bm{v} = \bm{v}$ since $\widetilde{P}^T \widetilde{P} = I$.  Note that
\begin{equation}\label{eqn:lower-c0}
\| \widetilde{\bm{v}} \|^2_{\widetilde{A}_D} = \langle \widetilde{A}_D \widetilde{P} \bm{v}, \widetilde{P} \bm{v} \rangle = \| \bm{v} \|_A^2,
\end{equation}
which implies condition \eqref{ine:lower-c0} holds with $c_0 = 1$.

To show that condition \eqref{ine:upper-c1} holds, we use the following result. 
\begin{lemma}\label{taylor}
Let $A\in \mathbb{R}^{n\times n}$ be a graph Laplacian corresponding to a connected graph with $n$ vertices. For all 
$i\in\{1,\ldots, n\}$ and $\bm{u} \in \mathbb{R}^n$ we have
\begin{equation*}
\frac{1}{n}(\bm{u},\bm{1}_n)-(\bm{u},\bm{e}_i) = \frac1n (A_i^{-1} \bm{1}_n, A \bm{u}),
\end{equation*}
where $A_i = (A+\bm{e}_i\bm{e}_i^T)$. 
\end{lemma}
\begin{proof}
  First, we note that for all $i\in \{1,\ldots,n\}$ the matrices $A_i$
  are invertible, because they all are irreducibly diagonally dominant
  $M$-matrices. We refer to  Varga~\cite{2000VargaR-aa} for this classical
  result. 

  Next, observe that $A_i \bm{1}_n = \bm{e}_i$, and hence,
  $A_i^{-1}\bm{e}_i = \bm{1}_n$.  Therefore, we have that
\begin{eqnarray*}
(A_i^{-1} \bm{1}_n, A \bm{u}) & = & 
(A_i^{-1} \bm{1}_n, (A_i-\bm{e}_i\bm{e}_i^T) \bm{u}) 
 =   
 (\bm{1}_n, \bm{u}) - (\bm{u},\bm{e}_i)(A_i^{-1} \bm{1}_n, \bm{e}_i)\\
& = &  
 (\bm{1}_n, \bm{u}) - 
(\bm{u},\bm{e}_i)(\bm{1}_n, A_i^{-1} \bm{e}_i)  = 
 (\bm{1}_n, \bm{u}) - (\bm{u},\bm{e}_i)(\bm{1}_n, \bm{1}_n).
\end{eqnarray*}
As $(\bm{1}_n, \bm{1}_n)=n$, this completes the proof. 
\end{proof}
The result shown in Lemma~\ref{taylor} is also found
in~\cite[Lemma~3.2]{Brannick.J;Chen.Y;Kraus.J;Zikatanov.L.2013a}, but is included for completeness.

We now apply Lemma~\ref{taylor} to each disaggregated local subgraph
$\mathsf{G}_a^k = (V_a^k, E_a^k, \omega^k_a)$, $k=1,2,\cdots, m$, with
$\bm{u} = \widetilde{\bm{v}}_k$, the restriction of
$\widetilde{\bm{v}}$ on $\mathsf{G}_a^k$. For $j \in V_a^k$, we
have,
\begin{equation}\label{eqn:taylor-dis-k}
  \widetilde{v}_j = \frac{1}{d_k} \sum_{p \in V_a^k} \widetilde{v}_p - \frac{1}{d_k} \langle L_{k,j}^{-1} \bm{1}_{d_k}, L_k \widetilde{\bm{v}}_k \rangle,
\end{equation}
where $L_k$ is the unweighted graph Laplacian of the local graph $\mathsf{G}_a^k$ 
and $L_{k,j}$ is defined in accordance with Lemma~\ref{taylor}: 
$L_{k,j} = L_k + \bm{e}_j^k \left( \bm{e}^k_j \right)^T$, for $j \in V_a^k$.  
Setting $W_k^j := \frac{1}{d_k^2} \| L_{k,j}^{-1} \bm{1}_{d_k} \|^2_{L_k} $, $j \in V_a^k$, and denoting  
\begin{align*}
E_D^0 &:= \{  e=(i,j) \in E_D, i, j \in V^0  \}, \\
E_D^1 &:= \{  e=(i,j) \in E_D, i\in V^0, j\in V_a^k, k=1,2,\cdots,m  \}, \\
E_D^2 &:= \{  e=(i,j) \in E_D, i \in V_a^k, j\in V_a^\ell, k, \ell=1,2,\cdots,m, 
k \neq \ell \},
\end{align*}

 we are ready to present the following lemma related to the condition \eqref{ine:upper-c1}. 

\begin{lemma} \label{lem:upper-c1}
For each disaggregated local subgraph $\mathsf{G}_a^k$, if, for an edge $e' = (p,q) \in E_a^k$, we assign a weight $\omega_{e'}$ such that
\begin{equation}\label{def:weight}
\omega_{e'} \geq W_{e'} := (1+\epsilon^{-1}) \left[ \sum_{\substack{e = (i,j) \in E_D^1 \\ i\in V^0, \ j\in V_a^k}}  \omega_e W_k^j   +  2 \sum_{\ell=1}^{m} \left(  \sum_{\substack{e=(i,j)\in E_D^2 \\  i \in V_a^k, \ j \in V_a^\ell }} \omega_e W_k^i  + \sum_{\substack{e=(i,j)\in E_D^2 \\  i \in V_a^\ell, \ j \in V_a^k }} \omega_e W_k^j  \right) \right],
\end{equation}
then we have
\begin{equation}\label{ine:upper-epsilon} 
\| \widetilde{P}^T \widetilde{\bm{v}} \|^2_A \leq (1+\epsilon) \| \widetilde{\bm{v}} \|^2_{\widetilde{A}_D}, \quad \forall \ \widetilde{v} \in \widetilde{V},
\end{equation}
where $\epsilon > 0$.
\end{lemma}
\begin{proof}
  We denote $\widetilde{\bm{u}} = \widetilde{P}^T \widetilde{P} \widetilde{\bm{v}}$, and we have,
\begin{eqnarray*}
\| \widetilde{P}^T \widetilde{\bm{v}} \|_A^2  & = & \langle \widetilde{A}_D \widetilde{\bm{u}}, \widetilde{\bm{u}} \rangle  = \sum_{e=(i,j) \in E_D} \omega_e ( \widetilde{u}_i - \widetilde{u}_j )^2 \\
& = & \sum_{e=(i,j)  \in E_D^0} \omega_e  ( \widetilde{u}_i - \widetilde{u}_j )^2 + \sum_{k=1}^m \sum_{e=(i,j)  \in E_a^k} \omega_e  ( \widetilde{u}_i - \widetilde{u}_j )^2  \\
&& + \sum_{e=(i,j)  \in E_D^1} \omega_e  ( \widetilde{u}_i - \widetilde{u}_j )^2 + \sum_{e=(i,j)  \in E_D^2} \omega_e  \left( \widetilde{u}_i - \widetilde{u}_j \right)^2 \\
& =: & I_0 + I_1 + I_2.
\end{eqnarray*} 
Here, we have set
$I_0=\sum_{e=(i,j)  \in E_D^0} \omega_e  \left( \widetilde{v}_i - \widetilde{v}_j \right)^2$, 
\[
I_1=\sum_{k=1}^m \sum_{\substack{e = (i,j) \in E_D^1 \\ i\in V^0, \ j\in V_a^k}} \omega_e  \left( \widetilde{v}_i -  \frac{1}{d_k}\sum_{p \in V_a^k} \widetilde{v}_p \right)^2,
\]
and 
\begin{eqnarray*} 
I_2&=& \sum_{k=1}^m \sum_{\ell=1}^m  \sum_{\substack{e=(i,j) \in E_D^2 \\ i \in V_a^k, \ j \in V_a^\ell}} \omega_e  \left( \frac{1}{d_k}\sum_{p \in V_a^k} \widetilde{v}_p - \frac{1}{d_\ell}\sum_{q \in V_a^\ell} \widetilde{v}_q \right)^2.
\end{eqnarray*}

Next, we estimate $I_1$ and $I_2$ on the right-hand side. For
$e=(i,j) \in E_D^1$, $i \in V^0$ and $j \in V_a^k$, using
\eqref{eqn:taylor-dis-k}, we have
\begin{align*}
\left( \widetilde{v}_i - \frac{1}{d_k}\sum_{p \in V_a^k} \widetilde{v}_p \right)^2 & = \left(  \widetilde{v}_i - \widetilde{v}_j -  \frac{1}{d_k} \langle L_{k,j}^{-1} \bm{1}_{d_k}, L_k \widetilde{\bm{v}}_k \rangle  \right)^2 \\
& \leq (1+ \epsilon) \left( \widetilde{v}_i - \widetilde{v}_j \right)^2 + \left( 1 + \epsilon^{-1} \right) \frac{1}{d_k^2} \| L_{k,j}^{-1} \bm{1}_{d_k} \|^2_{L_k} \| \widetilde{\bm{v}}_k \|^2_{L_k} \\
& = (1+ \epsilon) \left( \widetilde{v}_i - \widetilde{v}_j \right)^2 + \sum_{e'=(p,q) \in E_a^k} \left[ \left( 1+ \epsilon^{-1} \right) W_k^j  \right] \left( \widetilde{v}_p - \widetilde{v}_q \right)^2.
\end{align*}
Then
\begin{eqnarray*}
I_1 & \leq &\sum_{k=1}^m \sum_{\substack{e = (i,j) \in E_D^1 \\ i\in V^0, \ j\in V_a^k}} \omega_e \left\{ (1+ \epsilon) \left( \widetilde{v}_i - \widetilde{v}_j \right)^2 + \sum_{ e'=(p,q) \in E_a^k} \left[ \left( 1+ \epsilon^{-1} \right) W_k^j  \right] \left( \widetilde{v}_p - \widetilde{v}_q \right)^2 \right \} \\
& = &\left( 1 + \epsilon \right) \sum_{k=1}^m\sum_{\substack{e = (i,j) \in E_D^1 \\ i\in V^0, \ j\in V_a^k}} \omega_e \left( \widetilde{v}_i - \widetilde{v}_j \right)^2 \\
&& + \sum_{k=1}^m \sum_{e'=(p,q) \in E_a^k} \left[  (1+\epsilon^{-1}) \sum_{\substack{e = (i,j) \in E_D^1 \\ i\in V^0, \ j\in V_a^k}}  \omega_e W_k^j  \right] \left( \widetilde{v}_p - \widetilde{v}_q \right)^2.
\end{eqnarray*}
Next, using \eqref{eqn:taylor-dis-k}, 
for $e = (i,j) \in E^2_D$, $i \in V_a^k$ and $j \in V_a^\ell$
we have
\begin{align*}
\left( \frac{1}{d_k}\sum_{p \in V_a^k} \widetilde{v}_p  - \frac{1}{d_\ell}\sum_{q \in V_a^\ell} \widetilde{v}_q  \right)^2 &= \left(  \widetilde{v}_i - \widetilde{v}_j +  \frac{1}{d_k} \langle 
L_{k,i}^{-1} \bm{1}_{d_k}, L_k \widetilde{\bm{v}}_k \rangle  - \frac{1}{d_\ell} \langle L_{\ell,j}^{-1} \bm{1}_{d_\ell}, L_\ell \widetilde{\bm{v}}_\ell \rangle      \right)^2 \\
& \leq (1+\epsilon ) \left(  \widetilde{v}_i - \widetilde{v}_j \right)^2 + 2(1+\epsilon^{-1}) \frac{1}{d_k^2} \| L_{k,i}^{-1} \bm{1}_{d_k} \|^2_{L_k}  \| \widetilde{\bm{v}}_k \|^2_{L_k} \\
& \quad + 2(1+\epsilon^{-1}) \frac{1}{d^2_\ell} \| L_{\ell,j}^{-1} \bm{1}_{d_\ell} \|_{L_{\ell}} \| \widetilde{\bm{v}}_{\ell} \|_{L_\ell}^2 \\
& = (1+\epsilon ) \left(  \widetilde{v}_i - \widetilde{v}_j \right)^2 + \sum_{e'=(p,q) \in E_a^k} \left[ 2\left( 1+ \epsilon^{-1} \right) W_k^i  \right] \left( \widetilde{v}_p - \widetilde{v}_q \right)^2 \\
& \quad + \sum_{e'=(p,q) \in E_a^\ell} \left[2 \left( 1+ \epsilon^{-1} \right) W_\ell^j  \right] \left( \widetilde{v}_p - \widetilde{v}_q \right)^2.
\end{align*}
Then 
\begin{eqnarray*}
I_2 &\leq &
\sum_{k=1}^m \sum_{\ell=1}^m \sum_{\substack{e=(i,j) \in E_D^2 \\ i \in V_a^k, \ j \in V_a^\ell}} \omega_e \left \{  (1+\epsilon ) \left(  \widetilde{v}_i - \widetilde{v}_j \right)^2 + \sum_{e'=(p,q) \in E_a^k} \left[ 2\left( 1+ \epsilon^{-1} \right) W_k^i  \right] \left( \widetilde{v}_p - \widetilde{v}_q \right)^2 \right. \\
&& \left.  +  \sum_{e'=(p,q) \in E_a^\ell} \left[2 \left( 1+ \epsilon^{-1} \right) W_\ell^j  \right] \left( \widetilde{v}_p - \widetilde{v}_q \right)^2  \right\}.
\end{eqnarray*}
Therefore, we have
\begin{eqnarray*}
I_2&\le&  (1+\epsilon) \sum_{k=1}^{m} \sum_{\ell=1}^{m} \sum_{\substack{e=(i,j) \in E_D^2 \\ i \in V_a^k, \ j \in V_a^\ell}} \omega_e \left(  \widetilde{v}_i - \widetilde{v}_j \right)^2  \\
&&+ \sum_{k=1}^m \sum_{e'=(p,q) \in E_a^k} \left[  \sum_{l=1}^m \sum_{\substack{e=(i,j) \in E_D^2 \\ i \in V_a^k, \ j \in V_a^\ell}} 2(1+\epsilon^{-1})\omega_e W_k^i \right]  \left( \widetilde{v}_p - \widetilde{v}_q \right)^2  \\
&& + \sum_{\ell=1}^m \sum_{e'=(p,q) \in E_a^\ell} \left[  \sum_{k=1}^m \sum_{\substack{e=(i,j) \in E_D^2 \\ i \in V_a^k, \ j \in V_a^\ell}} 2 (1+\epsilon^{-1}) \omega_e W_\ell^j  \right]  \left( \widetilde{v}_p - \widetilde{v}_q \right)^2.
\end{eqnarray*}
Hence, 
\begin{eqnarray*}
I_2& \le & (1+\epsilon) \sum_{k=1}^{m} \sum_{\ell=1}^{m} \sum_{\substack{e=(i,j) \in E_D^2 \\ i \in V_a^k, \ j \in V_a^\ell}} \omega_e \left(  \widetilde{v}_i - \widetilde{v}_j \right)^2  \\
& & + \sum_{k=1}^{m} \sum_{e'=(p,q)\in E_a^k} \left[ 2 (1+\epsilon^{-1}) \left(  \sum_{\ell=1}^m \sum_{\substack{e=(i,j)\in E_D^2 \\  i \in V_a^k, \ j \in V_a^\ell }} \omega_e W_k^i  + \sum_{\ell=1}^{m} \sum_{\substack{e=(i,j)\in E_D^2 \\  i \in V_a^\ell, \ j \in V_a^k }} \omega_e W_k^j  \right)   \right] \left( \widetilde{v}_p - \widetilde{v}_q \right)^2.
\end{eqnarray*} 

Now, we use the definition of $W_{e'}$ \eqref{def:weight} and the estimates on $I_1$ and $I_2$ 
to obtain that
\begin{eqnarray*}
&& \| \widetilde{P}^T \widetilde{\bm{v}} \|_A^2  \leq \sum_{e \in E_D^0} \omega_e \left(  \widetilde{v}_i - \widetilde{v}_j \right)^2  + \left( 1 + \epsilon \right) \sum_{k=1}^m\sum_{\substack{e = (i,j) \in E_D^1 \\ i\in V^0, \ j\in V_a^k}} \omega_e \left( \widetilde{v}_i - \widetilde{v}_j \right)^2  \\
&&\quad + (1+\epsilon) \sum_{k=1}^{m} \sum_{\ell=1}^{m} \sum_{\substack{e=(i,j) \in E_D^2 \\ i \in V_a^k, \ j \in V_a^\ell}} \omega_e \left(  \widetilde{v}_i - \widetilde{v}_j \right)^2  +  \sum_{k=1}^{m} \sum_{e'=(p,q)\in E_a^k} W_{e'} \left( \widetilde{v}_p - \widetilde{v}_q \right)^2.
\end{eqnarray*}
Due to \eqref{def:weight}, we have that $\omega_{e'} \geq W_{e'}$ 
and \eqref{ine:upper-epsilon} follows. This completes the proof. 
\end{proof}

Lemma \ref{lem:upper-c1} shows that the constant $c_1$ can be made
arbitrarily close to $1$ if the weights on the internal edges of the
disaggregation are chosen to be large enough.  As an immediate
consequence, we have the following theorem for the preconditioner $B$.
\begin{theorem}\label{thm:prec-B}
  Under the assumptions of Corollary \ref{coro:prec-B} and
  Lemma~\ref{lem:upper-c1}, for the preconditioner
  $B = \widetilde{P}^T \widetilde{B}_D \widetilde{P}$, we have
\begin{equation} \label{ine:prec-B-cond}
\kappa(BA) \leq \left( 1 + \epsilon \right) \kappa(\widetilde{B}_D \widetilde{A}_D).
\end{equation}
\end{theorem} 
\begin{proof}
  The relation~\eqref{ine:prec-B-cond} follows from
  Corollary~\ref{coro:prec-B} since $c_0 = 1$ in~\eqref{eqn:lower-c0}
  and $c_1 = (1+\epsilon)^{1/2}$ in Lemma~\ref{lem:upper-c1}.
\end{proof}

Finally, since $\widetilde{A}_D:= D_s^{-1} A_D D_s^{-1}$, if we have a
preconditioner $B_D$ for $A_D$ and define
$\widetilde{B}_D = D_s B_D D_s$, then it is easy to verify that
$\kappa(\widetilde{B}_D\widetilde{A}_D) = \kappa(B_DA_D)$. We have the
following theorem showing that the preconditioned operator $BA$ has a
condition number comparable to the condition number of $B_DA_D$.
\begin{theorem}\label{thm:prec-B1}
Under the assumptions of Corollary \ref{coro:prec-B} and Lemma \ref{lem:upper-c1} and let $\widetilde{B}_D = D_s B_D D_s$, for the preconditioner $B = \widetilde{P}^T \widetilde{B}_D \widetilde{P}$, we have
 \begin{equation} \label{ine:prec-B1-cond}
\kappa(BA) \leq \left( 1 + \epsilon \right) \kappa(B_D A_D).
\end{equation}
\end{theorem}
\begin{proof}
\eqref{ine:prec-B1-cond} follows from Theorem \ref{thm:prec-B} and the fact that $\kappa(\widetilde{B}_D\widetilde{A}_D) = \kappa(B_DA_D)$.
\end{proof}
Clearly, Theorems \ref{ine:prec-B-cond} and \ref{ine:prec-B1-cond} imply that,
when the weights on the internal edges of the disaggregation are chosen
to be large enough, preconditioners for disaggregated graph provide
effective preconditioners for the original graph, which indirectly
supports the technique suggested in~\cite{Kuhlemann.V;Vassilevski.P2013a}.

\section*{Acknowledgements}

The authors would like to thank Louisa Thomas for improving the style
of the presentation. The work of Ludmil Zikatanov was supported in
part by the National Science Foundation under grants DMS-1418843 and
DMS-1522615 and by the Department of Mathematics at Tufts University.

%
%

\bibliographystyle{plain}
\bibliography{disagg_bib} 

\end{document}